\documentclass{amsart}
\usepackage{amssymb}

\newtheorem{theorem}{Theorem}
\newtheorem{remark}[theorem]{Remark}
\newtheorem{lemma}[theorem]{Lemma}

\newtheorem{prop}[theorem]{Proposition}
\newtheorem{example}[theorem]{Example}
\newtheorem{notation}[theorem]{Notation}
\newtheorem{assumption}[theorem]{Assumption}
\numberwithin{equation}{section}
\numberwithin{theorem}{section}

\def\P{{\mathbb P}}        
\def\E{{\mathbb E}}        

\title[Small deviations and
second-order chaos]
{Small deviations for time-changed Brownian motions and applications to
second-order chaos}
\author[Dobbs]{Daniel Dobbs}
\address{Department of Mathematics\\
Huntington University \\
Huntington, IN 46750 USA} \email{ddobbs@huntington.edu}
\author[Melcher]{Tai Melcher{$^*$}}
\thanks{\footnotemark {$^*$} This research was supported in part by NSF
Grants DMS-0907293 and DMS-1255574.}
\address{Department of Mathematics\\
University of Virginia\\ Charlottesville, VA 22903 USA}
\email{melcher@virginia.edu}


\begin{document}

\begin{abstract}
We prove strong small deviations results for Brownian motion under independent
time-changes satisfying their own asymptotic criteria.
We then apply these results to certain stochastic integrals which are
elements of second-order homogeneous chaos.
\end{abstract}

\maketitle
\tableofcontents

\section{Introduction}\label{s.intro}

In this paper, we study small deviations for some time-changed Brownian
motions, for the purpose of applications to
certain elements of Wiener chaos.
Large deviation estimates for Wiener chaos are well-studied (see for example
\cite{Ledoux1994}), largely due to the work of Borell (see for example
\cite{Borell1978} and \cite{Borell1984}).  However, small deviations in this
setting are much less understood and are of interest for their myriad
interactions with other concentration, comparison, and correlation
inequalities as well as various limit laws for stochastic processes; see for example the surveys \cite{LiShao2001} and
\cite{Lifshits1999}.
The present work gives strong small
deviations results for certain elements of second-order homogeneous chaos.
In particular, let $(\mathcal{W},\mathcal{H},\mu)$ be an abstract Wiener
space, $\{W_t\}_{t\ge0}$ denote Brownian motion on $\mathcal{W}$, and
$\omega:\mathcal{W}\times\mathcal{W}\rightarrow\mathbb{R}$ be a continuous
bilinear antisymmetric map .  We will study processes $\{Z(t)\}_{t\ge0}$ of the form
\begin{equation}
\label{e.zt}
Z(t) := \int_0^t \omega(W_s,dW_s).
\end{equation}
(A precise definition is given in Section \ref{s.defZ}.)
In particular, we show that $Z$ is equal in distribution to
a Brownian motion running on an independent random clock for which
small deviation probabilities are known, and thus the small deviations
behavior of $Z$ follows.  From these results one may infer, for example, a functional law of iterated
logarithm and hence a Chung-type law of iterated logarithm for $Z$.  To the authors'
knowledge, these are the first results for small deviations of elements of
Wiener chaos in the infinite-dimensional context beyond the first-order
Gaussian case.

\subsection{Statement of main results}

We first discuss the general small deviations result for time-changed Brownian
motion we will be using.  We will assume that the random clocks satisfy the
following.

\begin{assumption}
\label{a.a}
Suppose $\{C(t)\}_{t\ge0}$ is a continuous non-negative non-decreasing process such that
$C(0)=0$ and there exist $\alpha>0$, $\beta\in\mathbb{R}$, and a non-decreasing function
$K:(0,\infty)\rightarrow(0,\infty)$ such that
for any $m\in\mathbb{N}$, $\{d_i\}_{i=1}^m\subset(0,\infty)$ a decreasing
sequence, and $0=t_0< t_1<\cdots <t_m$,
\begin{equation}
\label{e.PC}
\lim_{\varepsilon\downarrow 0}
		\varepsilon^{\alpha}|\log\varepsilon|^{\beta} \log
		\P\left(\sum_{i=1}^m d_i \Delta_i C \le \varepsilon\right)
	= - \left( \sum_{i=1}^m \left(d_i^\alpha K(t_{i-1},t_i)
		\right)^{1/(1+\alpha)}\right)^{(1+\alpha)}
\end{equation}
where $\Delta_i C = C_{t_i}-C_{t_{i-1}}$.
\end{assumption}
By the exponential Tauberian theorem (see Theorem \ref{t.tauberian}), equation (\ref{e.PC}) is equivalent to
\begin{multline}
\label{e.EC}
\lim_{\lambda\rightarrow\infty}
	\lambda^{-\alpha/(1+\alpha)} (\log\lambda)^{\beta/(1+\alpha)}
		\log\E\left[\exp\left(-\lambda\sum_{i=1}^m d_i \Delta_i C \right)\right] \\
	= - (1+\alpha)^{1+\beta/(1+\alpha)}\alpha^{-\alpha/(1+\alpha)}\sum_{i=1}^m \left(d_i^\alpha K(
		t_{i-1},t_i)\right)^{1/(1+\alpha)}.
\end{multline}
Also via the exponential Tauberian theorem, equation (\ref{e.PC}) clearly
holds when $C$ is a subordinator satisfying
\[ \lim_{\varepsilon\downarrow 0}
		\varepsilon^{\alpha}|\log\varepsilon|^{\beta}
		\log \P(C(t) \le \varepsilon)
	= -K(t) \]
for any $t>0$, although the additional requirement of continuity makes this example trivial (since in this case $C(t)=ct$ a.s.~for some $c\ge0$).
More generally, (\ref{e.PC}) holds if $C$ has independent increments which satisfy
\[ \lim_{\varepsilon\downarrow 0}
		\varepsilon^{\alpha}|\log\varepsilon|^{\beta}
		\log \P(C(t)-C(s) \le \varepsilon)
	= -K(s,t) \]
for all $0\le s<t$.  However, it is not necessary for $C$ to have independent or
stationary increments for Assumption \ref{a.a} to hold.  One important source of
examples for the present paper is the following theorem from \cite{Li2001} for
weighted $L^p$ norms of a Brownian motion.

\begin{theorem}
\label{t.6.4}
Let $p\in[1,\infty)$ and $\rho:[0,\infty)\rightarrow[0,\infty]$ be a Lebesgue measurable function
satisfying
\begin{itemize}
\item[(i)] $\rho$ is bounded or non-increasing on $[0,a]$ for some $a>0$;
\item[(ii)] $t^{(2+p)/p}\rho(t)$ is bounded or non-decreasing on $[A,\infty)$
for some $A<\infty$;
\item[(iii)] $\rho$ is bounded on $[a,A]$; and
\item[(iv)]  $\rho^{2p/(p+2)}$ is Riemann integrable on $[0,\infty)$.
\end{itemize}
Then
\[
\lim_{\varepsilon\downarrow0} \varepsilon^{2/p}\log\P\left(\int_0^\infty
		\rho(s)^p|B(s)|^p\,ds\le\varepsilon\right)
	= -\kappa_p\left(\int_0^\infty\rho(s)^{2p/(2+p)}\,ds\right)^{(2+p)/p},
\]
where $\kappa_p= 2^{2/p}p\left(\frac{\lambda_1(p)}{2+p}\right)^{(2+p)/2}$ for
\[ \lambda_1(p)
	= \inf_{ {\tiny \begin{array}{cc} \phi\in L^2(-\infty,\infty) \\
		\|\phi\|=1\end{array}}}
	\left\{\int_{-\infty}^\infty |x|^p\phi^2(x)\,dx +
		\frac{1}{2}\int_{-\infty}^\infty (\phi'(x))^2\,dx \right\}. \]
\end{theorem}
For example, if $\tilde{\rho}$ is any non-negative continuous function
on $[0,\infty)$ and
\[ C(t) = \int_0^t \tilde{\rho}(s)^p|B(s)|^p\,ds, \]
then
\[ \sum_{i=1}^m d_i\Delta_i C
	= \sum_{i=1}^m d_i \int_{t_{i-1}}^{t_i} \tilde{\rho}(s)^p|B(s)|^p\,ds \]
and applying Theorem \ref{t.6.4} with $\rho(s)=\sum_{i=1}^m
d_i^{1/p} 1_{(t_{i-1},t_i]}(s) \tilde{\rho}(s)$ gives (\ref{e.PC}) with
$\alpha=2/p$, $\beta=0$, and
\[ K(t_{i-1},t_i)= \left(\int_{t_{i-1}}^{t_i}
		\tilde{\rho}(s)^{2p/(p+2)}\,ds\right)^{(2+p)/p}. \]

A particularly relevant example to our later applications is the simplest
case where $p=2$ and $\tilde{\rho}\equiv 1$, for which
\begin{equation}
\label{e.basic}
C(t) = \int_0^t B(s)^2\,ds,
\end{equation}
$\kappa_2=1/8$, and $K(t_{i-1},t_i)=(\Delta_i t)^2$ where $\Delta_i
t:=t_i-t_{i-1}$.

See Section 6 of \cite{LiShao2001} for more results related to Theorem
\ref{t.6.4}. Additionally, Chapter 7.3 of \cite{LifshitsLinde2002} contains these results under weaker assumptions. Known small deviations for weighted
$L^p$ norms of other stochastic processes
provide other interesting examples.  For example, in
\cite{LifshitsLinde2005} a result analogous to Theorem \ref{t.6.4} is
proved for fractional Brownian motions.  See this and related references for
further examples.

Now working under Assumption \ref{a.a}, one may prove the
following.

\begin{theorem}
\label{t.SB}
Suppose that $\{Z(t)\}_{t\ge0}$ is a stochastic process given by
$Z(t)=B(C(t))$, where $C$ is as in Assumption \ref{a.a} and
$B$ is a standard real-valued Brownian motion independent
of $C$.
Let $M(t):=\sup_{0\le s\le t} |Z(s)|$.  Then, for any $m\in\mathbb{N}$, $0=t_0<t_1<\cdots<t_m<\infty$, and $0\le
a_1<b_1\le a_2<b_2\le \cdots \le a_m<b_m$,
\begin{multline*}
    \lim_{\varepsilon\downarrow 0}
        \varepsilon^{2\alpha/(1+\alpha)}|\log \varepsilon|^{\beta/(1+\alpha)}
		\log \P\left(\bigcap_{i=1}^m
                            \{a_i\varepsilon\leq M(t_i)\leq b_i\varepsilon\}
                        \right) \\
                    =
                        -2^{-\beta/(1+\alpha)}(1+\alpha)^{1+\beta/(1+\alpha)}
		\left(\frac{\pi^2}{8\alpha}\right)^{\alpha/(1+\alpha)}
		\sum_{i=1}^m \left(\frac{K(t_{i-1},
		t_i)}{b_i^{2\alpha}}\right)^{1/(1+\alpha)}.
\end{multline*}
\end{theorem}
Such estimates have been previously studied for processes $\{Z_t\}_{t\ge0}$ that are
symmetric $\alpha$-stable processes \cite{ChenKuelbsLi2000}, fractional
Brownian motions \cite{KuelbsLi2002}, certain stochastic
integrals \cite{KuelbsLi2005}, $m$-fold integrated Brownian
motions \cite{ZhangLin2006}, and integrated $\alpha$-stable
processes \cite{ZhangLin2010}.  In particular, the stochastic integrals studied in
\cite{KuelbsLi2005} are essentially finite-dimensional versions of
the class of stochastic integral processes we study, and the proof that we
give for Theorem \ref{t.SB} follows the outline of the proof of small ball
estimates in that reference.

We apply Theorem \ref{t.SB} to stochastic integrals of the form (\ref{e.zt}) as follows.
\begin{theorem}
\label{t.introz}
Let $\{Z(t)\}_{t\ge0}$ be as in equation (\ref{e.zt}).
Then $\{Z(t)\}\overset{d}{=}\{B(C(t))\}$ for $B$ a standard real-valued
Brownian motion and
\[ C(t) = 
\sum_{k=1}^\infty \|\omega(e_k,\cdot)\|_{\mathcal{H}}^2 \int_0^t
(W_s^k)^2\,ds
\]
where $\{e_k\}_{k=1}^\infty$ is any orthonormal basis of $\mathcal{H}$
contained in $\mathcal{H}_*:= \{h\in\mathcal{H}:\langle h,\cdot\rangle$ extends to a continuous linear functional on $\mathcal{W}\}$
and $\{W^k\}_{k=1}^\infty$ are independent standard Brownian motions which
are also independent of $B$.
If we further suppose that
$\|\omega(e_k,\cdot)\|_{\mathcal{H}}=O(k^{-r})$ for
$r>1$,
then, for any $m\in\mathbb{N}$,
$0=t_0<t_1<\cdots<t_m$, and $\{d_i\}_{i=1}^m\subset(0,\infty)$ a decreasing
sequence,
\[ \lim_{\varepsilon\downarrow0} \varepsilon \log\P\left(\sum_{i=1}^m
		d_i^2 \Delta_i C\le\varepsilon \right)
	= -\frac{1}{2}\|\omega\|_1^2 \left(\sum_{i=1}^m d_i\Delta_it\right)^2, \]
where $\Delta_it:=t_i-t_{i-1}$ and
\[ \|\omega\|_1 := \sum_{k=1}^\infty \|\omega(e_k,\cdot)\|_{\mathcal{H}}
<\infty. \]
Thus, for any $0\leq a_1<b_1\leq a_2<b_2\leq\cdots\leq a_m<b_m$,
    \begin{equation*}
      \lim_{\varepsilon\downarrow0}
            \varepsilon\log P\left( \bigcap_{i=1}^m
                               \{ a_i\varepsilon\leq \sup_{0\leq s\leq
t_i}|Z_s|\leq b_i\varepsilon \}                            \right)
                        =
                            -\frac{\pi}{4}\|\omega\|_1\sum_{i=1}^m\frac{\Delta_i
t}{b_i}.
    \end{equation*}
\end{theorem}

\begin{remark}
Note that in the above theorem, and in the sequel, we make the standard identification between $\mathcal{H}^*$ and $\mathcal{H}$ via Riesz representation.  That is, for a linear functional $\varphi$ on $\mathcal{H}$, we write
\[ \|\varphi\|_\mathcal{H}^2 = \sum_{j=1}^\infty |\langle \varphi,e_j\rangle_\mathcal{H}|^2
	= \sum_{j=1}^\infty |\varphi(e_j)|^2. \]
In particular, for fixed $h\in\mathcal{H}$, we write $\|\omega(h,\cdot)\|_\mathcal{H}^2 := \|\omega(h,\cdot)|_\mathcal{H}\|_\mathcal{H}^2=\sum_{j=1}^\infty |\omega(h,e_j)|^2$.
\end{remark}

Applications of such estimates include using the small deviations in Theorem \ref{t.introz} to prove a
Chung-type law of iterated logarithm as well as a functional law
of iterated logarithm for the process $Z$.  We record these results in
Theorem \ref{t.chungLIL} and \ref{t.FLIL}.

\subsection{Discussion}
First-order small deviation estimates of the standard form
\[
\log \P\left(\sup_{0\le s\le t} |Z(s)| \le \varepsilon\right)
\]
were studied in \cite{Remillard1994} for processes $Z(t)=\int_0^t
\omega(W_s,dW_s)$  with $W$ an $n$-dimensional Brownian motion and
$\omega:\mathbb{R}^n\rightarrow\mathbb{R}$ given by $\omega(x,y)=Ax\cdot
y$ for $A$ a skew-symmetric $n\times n$ matrix.  These estimates
were then applied to prove an analogue of the classical limit result
of Chung.  (This was done earlier in \cite{Shi1994} in the case
$n=2$ and $A=\begin{pmatrix} 0 &1/2\\-1/2&0\end{pmatrix}$, that is,
for $Z$ the stochastic L\'evy area.) In \cite{KuelbsLi2005}, the
authors improved these results by proving stronger asymptotic
results like those in Theorem \ref{t.SB} for the same $Z$ as in
\cite{Remillard1994} and
applying these results to prove a functional law of iterated
logarithm.

In the present paper, the proof of the small ball estimates established
in Theorem \ref{t.SB} is a direct generalization of the techniques of
\cite{KuelbsLi2005}.  However, Theorem \ref{t.SB} is sufficiently
general to be of independent interest for other potential applications.
Thus for that purpose, as well as for clarity and self-containment,
we include the proof here.  It is also clear from the proofs that, given only the
asymptotic order for $C$, one could infer the asymptotic order for
$Z$ instead.

We also mention the reference \cite{AurzadaLifshits2009}, in which
the authors study general iterated processes of the form $X\circ
Y$ where $X$ is a two-sided self-similar process and $Y$ is a
continuous process independent of $X$.  Since $X$ is two-sided, it
is not required that $Y$ satisfy any monotonicity or positivity
criteria.  In this general setting, under the assumption that the
first-order ($m=1$) asymptotics are known for $X$ and $Y$, the authors
are able to prove a first-order small ball estimate (Theorem
4 of \cite{AurzadaLifshits2009}).  Theorem \ref{t.SB} is stated in
the restricted setting that $X$ is a Brownian motion; however, the
proof carries through for first-order estimates
for processes $X$ satisfying more general assumptions
(as in \cite{AurzadaLifshits2009}).  See Proposition \ref{p.au} for
more details.

The organization of the paper is as follows.  In Section \ref{s.SB}
we give the proof of Theorem \ref{t.SB}.  In Section \ref{s.defZ},
we apply Theorem \ref{t.SB} to prove small ball estimates for the
relevant collection of stochastic integrals.  In Section \ref{s.defZ},
we define precisely the processes of interest, and in Theorem
\ref{t.Representation} we prove that these processes have a
representation as Brownian motions on an independent random clock.
In Subsection \ref{ss.smallBall}, we determine the small ball asymptotics
of the clock.  Thus we are able to apply Theorem \ref{t.SB}, and
we additionally record a Chung-type law of iterated logarithm and
functional law of iterated logarithm that follow from these estimates.

{\it Acknowledgement.} This paper is dedicated to the memory of Wenbo Li, who
suggested the problems addressed in Section \ref{s.defZ} of this paper, thus
motivating the whole of this work.

The authors would also like to thank an anonymous referee for careful reading and several useful comments to improve this paper.

\section{Small deviation estimates}
\label{s.SB}

In this section, we prove separately the upper and lower bounds of Theorem
\ref{t.SB}. The outline of the proof here follows Section 4 of
\cite{KuelbsLi2005}.
First, we record a standard
relation between asymptotics of the Laplace transform and small
ball estimate of a positive random variable in the form of
the exponential Tauberian theorem (see for
example Theorem 4.12.9 in \cite{BGT1987}).
We give a special case of that theorem here.
\begin{theorem}
\label{t.tauberian}
Suppose that $X$ is a positive random variable.  There exist $\alpha>0$,
$\beta\in\mathbb{R}$, and $K\in(0,\infty)$ such that
\begin{equation*}
\lim_{\varepsilon\downarrow 0}
	\varepsilon^{\alpha}|\log\varepsilon|^{\beta} \log \P(X \le \varepsilon)
	= - K
\end{equation*}
if and only if
\[ \lim_{\lambda\rightarrow\infty}
	\lambda^{-\alpha/(1+\alpha)}(\log\lambda)^{\beta/(1+\alpha)}
	\log\E[e^{-\lambda X}]
	= -(1+\alpha)^{1+\beta/(1+\alpha)} (\alpha^{-\alpha}K)^{1/(1+\alpha)}.
\]
\end{theorem}
We will use this theorem repeatedly in the sequel along with the standard fact
that, for any $\varepsilon>0$,
\begin{equation}
\label{e.bmb}
    \frac{2}{\pi}e^{-\frac{\pi^2}{8\varepsilon^2}}
        \leq \P\left(
                \sup_{0\leq s\leq 1}|B(s)|\leq \varepsilon
            \right)
        \leq\frac{4}{\pi}e^{-\frac{\pi^2}{8\varepsilon^2}},
\end{equation}
see for example \cite{Chung1948}.
Now the upper bound of Theorem \ref{t.SB} follows almost immediately
from this and the upper bound for the random clock $C$ via conditioning.

\begin{notation}
For $C$ as in Theorem \ref{t.SB}, we let $P_C(\cdot)=\P(\cdot\mid C)$.
\end{notation}

\begin{prop}\label{p.kuelbsLiProp1}
Under the hypotheses of Theorem \ref{t.SB}, we have that
\begin{multline*}
    \limsup_{\varepsilon\downarrow 0}
        \varepsilon^{2\alpha/(1+\alpha)}|\log \varepsilon|^{\beta/(1+\alpha)}
		\log \P\left(\bigcap_{i=1}^m
                            \{a_i\varepsilon\leq M(t_i)\leq b_i\varepsilon\}
                        \right) \\
                    \leq
                        -2^{-\beta/(1+\alpha)}(1+\alpha)^{1+\beta/(1+\alpha)}
		\left(\frac{\pi^2}{8\alpha}\right)^{\alpha/(1+\alpha)}
		\sum_{i=1}^m \left(\frac{K(t_{i-1},
		t_i)}{b_i^{2\alpha}}\right)^{1/(1+\alpha)}.
\end{multline*}
\end{prop}

\begin{proof} We will show that
\begin{equation}\label{e.step1}
    \P\left(\bigcap_{i=1}^m
        \{a_i\varepsilon\leq M(t_i)\leq b_i\varepsilon\}
        \right)
        \leq\left(
                \frac{4}{\pi}
            \right)^m
            \mathbb{E}
                \left[
                    \exp\left(-\frac{\pi^2}{8\varepsilon^2}\sum_{i=1}^m\frac{\Delta_i{C}}{b_i^2}\right)
                \right].
\end{equation}
Then applying equation (\ref{e.EC}) with $d_i=1/b_i^2$ finishes the proof.
So first we define
\[
    A_i:=\left\{
            \sup_{t_{i-1}\leq s\leq t_i} \left|Z(s)\right|
        \leq
            b_i\varepsilon
        \right\}.
\]
Then we have that
\[ P_C\left(\bigcap_{i=1}^m
        \{a_i\varepsilon\leq M(t_i)\leq b_i\varepsilon\}
    \right)
    \leq P_C\left(
        \bigcap_{i=1}^m A_i
    \right), \]
and for $\mu_{C,t_{m-1}}(\cdot)=P_C\left(
                                Z(t_{m-1})\in\cdot
                            \right)$
\begin{align*}
 &   P_C\left(
        \bigcap_{i=1}^m A_i
    \right) \\
        &=   \int_\mathbb{R} P_C
            \left(
                \bigcap_{i=1}^{m-1} A_i,
                \sup_{t_{m-1}\leq s\leq t_m} | Z(s)-Z(t_{m-1})
                + x\big|\leq b_m\varepsilon\,\bigg|\, Z(t_{m-1})=x
            \right) \,d\mu_{C,t_{m-1}}(x)   \\
        &=  \int_\mathbb{R} P_C
            \left(\left.
                \bigcap_{i=1}^{m-1} A_i\,\right| Z(t_{m-1})=x
            \right) \\
        &\qquad\qquad\times P_C
            \left(
                \sup_{t_{m-1}\leq s\leq t_m}
                | Z(s)-Z(t_{m-1})
                + x|\leq b_m\varepsilon
            \right)\,d\mu_{C,t_{m-1}}(x)
     \end{align*}
since $\displaystyle \sup_{t_{m-1}\leq s\leq
t_m}|Z(s)-Z(t_{m-1})+x|$ is $P_C$ independent of
$Z(t_{m-1})$ and $\displaystyle \bigcap_{i=1}^{m-1}A_i$ by the $P_C$
independent increments of $Z.$

Since $\left\{Z(t)\right\}_{t\geq 0}$ is a $P_C$ Gaussian centered
process, we have by Anderson's inequality (see, for example Theorem 1.8.5 of
\cite{Bogachev1998}) that
\begin{align*}
    P_C\bigg(
            \sup_{t_{m-1}\leq s\leq t_m} &\left| Z(s)- Z(t_{m-1})+x\right|\leq b_m\varepsilon
        \bigg) \\
        &\leq   P_C\left(
                \sup_{t_{m-1}\leq s\leq t_m} \left| Z(s)- Z(t_{m-1})\right|\leq b_m\varepsilon
            \right) \\
            &=  P_C\left(
                    \sup_{0\leq s\leq 1} \left|
                                            B(s)
                                        \right|
                    \leq
                        \frac{b_m\varepsilon}{\sqrt{\Delta_m{C}}}
                \right),
    \end{align*}
by the monotonicity and continuity of $C$ and the stationary and scaling properties of Brownian motion. Thus
\[
    P_C\left(
            \bigcap_{i=1}^m A_i
        \right)
        \leq
            P_C\left(
                \bigcap_{i=1}^{m-1} A_i
            \right)
            P_C\left(
                    \sup_{0\leq s\leq 1}\left|
                                            B(s)
                                        \right|
                    \leq
                        \frac{b_m\varepsilon}{\sqrt{\Delta_m{C}}}
                \right).
\]
By iterating the above computation $m$ times we see that
\begin{align*}
    P_C\left(
        \bigcap_{i=1}^m A_i
    \right)
    &\leq
        \prod_{i=1}^m P_C\left(
                            \sup_{0\leq s\leq
1}|B(s)|\leq\frac{b_i\varepsilon}{\sqrt{\Delta_i{C}}}
                        \right) \\
        &\leq \left(\frac{4}{\pi}\right)^m
                \exp\left(-\frac{\pi^2}{8\varepsilon^2}\sum_{i=1}^m\frac{\Delta_i{C}}{b_i^2}\right)
\end{align*}
where the second inequality follows from the upper bound in (\ref{e.bmb}).
Taking the expectation of both sides yields \eqref{e.step1}.
\end{proof}


We now move towards obtaining the lower bounds with the following lemma.

\begin{lemma}\label{p.kuelbsLiProp2}
Fix $\gamma>0$, and let $0<\delta<\gamma$ be such that
$a_i(1+\delta)<b_i(1-\delta)$. Also let $f_i=f_i(\varepsilon,\delta)$ and
$g_i=g_i(\varepsilon,\delta)$ be given by
\[ f_i :=  P_C
                \left(
                    \sup_{0\leq s\leq 1}
                    |B(s)|\leq\frac{b_i(1-\delta)\varepsilon}{\sqrt{\Delta_i{C}}}
                \right) \text{ and }
 g_i := P_C
                \left(
                    \sup_{0\leq s\leq 1}
                    |B(s)|\leq\frac{a_i(1+\delta)\varepsilon}{\sqrt{\Delta_i{C}}}
                \right)\]
and set
\begin{equation*}
    \Phi:=\left\{ \phi=
        \{\phi_i\}_{i=1}^m:\phi_i\in\{f_i,g_i\}\,\mbox{and at least one $\phi_i=g_i$}
    \right\}.
\end{equation*}
Then
\begin{multline*}
      \P\left(\bigcap_{i=1}^m
        \{a_i\varepsilon\leq M(t_i)\leq b_i\varepsilon\}, |Z(t_m)|\leq b_m\gamma\varepsilon
    \right)  \\
    \geq  \mathbb{E}\left[
                    \prod_{i=1}^m f_iP_C\left(
                            |B(1)|\leq\frac{\Delta_i{b}\delta\varepsilon}{\sqrt{\Delta_i{C}}}
                        \right)
                \right]
                -\sum_{\phi\in\Phi}\mathbb{E}\left[
                        \prod_{i=1}^m \phi_i
                    \right],
\end{multline*}
where $\Delta_i{b}=b_i-b_{i-1}$ with $b_0=0$.
\end{lemma}

\begin{proof} Define
\[
    \Upsilon_i:=\left\{
            a_i\varepsilon\leq\sup_{t_{i-1}\leq s\leq t_i} \left|
                                                             Z(s)
                                                        \right|
                                                    \leq
                                                        b_i\varepsilon,
                                                        \left|
                                                             Z(t_i)
                                                        \right|
                                                    \leq
                                                        b_i\delta\varepsilon
    \right\}.
\]
Since $\displaystyle \sup_{t_{i-1}\leq s\leq t_i}
| Z(s)|\leq M(t_i)$ and $b_i\gamma\varepsilon\geq b_i\delta\varepsilon$
for all $i$, we have that
\[
    \bigcap_{i=1}^m \{
        a_i\varepsilon\leq  M(t_i)\leq b_i\varepsilon\}
           \cap \{| Z(t_m)|\leq b_m\gamma\varepsilon\}
        \supset
            \bigcap_{i=1}^m \Upsilon_i.
\]
Define
\begin{multline*}
    A_i:=\bigg\{
            a_i(1+\delta)\varepsilon
                \leq\sup_{t_{i-1}\leq s\leq t_i}\left|
                                                     Z(s)- Z(t_{i-1})
                                                \right|
                    \leq b_i(1-\delta)\varepsilon,  \\
            \left|
                 Z(t_i)- Z(t_{i-1})
            \right|
                \leq
                    (\Delta_i{b})\delta\varepsilon
        \bigg\}.
\end{multline*}
By $P_C$ independent increments, \begin{align*}
       P_C\left(
            \bigcap_{i=1}^m \Upsilon_i
        \right)
        &\geq P_C\left(
                    \bigcap_{i=1}^{m-1} \Upsilon_i\cap A_m
                \right)
        =     P_C\left(
                    \bigcap_{i=1}^{m-1} \Upsilon_i
                \right)P_C\left(
                            A_m
                        \right),
    \end{align*} and repeating this computation $m$ times gives that
    \[
        P_C\left(
            \bigcap_{i=1}^m \Upsilon_i
        \right)
            \geq \prod_{i=1}^m P_C\left(A_i\right).
    \]
Again we use the stationary and
scaling properties of Brownian motion, as well as \v{S}idak's Lemma (see for
example, Corollary 4.6.2 of \cite{Bogachev1998}), to show that
\begin{align*}
    P_C\left(A_i\right)
        &=P_C\left(
                \frac{a_i(1+\delta)\varepsilon}{\sqrt{\Delta_i{C}}}
                    \leq
                \sup_{0\leq s\leq 1}\left|B(s)\right|
                    \leq
                \frac{b_i(1-\delta)\varepsilon}{\sqrt{\Delta_i{C}}},
                |B(1)|\leq\frac{\Delta_i{b}\delta\varepsilon}{\sqrt{\Delta_i{C}}}
            \right) \\
    &\geq P_C
        \left(
            \frac{a_i(1+\delta)\varepsilon}{\sqrt{\Delta_i{C}}}
            \leq\sup_{0\leq s\leq 1}|B(s)|
            \leq\frac{b_i(1-\delta)\varepsilon}{\sqrt{\Delta_i{C}}}
        \right)
    P_C
        \left(
            |B(1)|\leq\frac{\Delta_i{b}\delta\varepsilon}{\sqrt{\Delta_i{C}}}
        \right) \\
	&= (f_i-g_i)  P_C
        \left(
            |B(1)|\leq\frac{\Delta_i{b}\delta\varepsilon}{\sqrt{\Delta_i{C}}}
        \right) .
\end{align*}
Thus, taking expectations we have that
\begin{align*}
\P\bigg(
        \bigcap_{i=1}^m \Upsilon_i
    \bigg)
    &\geq   \mathbb{E}\left[
                    \prod_{i=1}^m f_iP_C\left(
                                |B(1)|\leq\frac{\Delta_i{b}\delta\varepsilon}{\sqrt{\Delta_i{C}}}
                            \right)
                    -\sum_\Phi \prod_{i=1}^m \phi_iP_C\left(
                        |B(1)|\leq\frac{\Delta_i{b}\delta\varepsilon}{\sqrt{\Delta_i{C}}}
                    \right)
                \right] \\
    &
	\geq   \mathbb{E}\left[
                    \prod_{i=1}^m f_iP_C\left(
                            |B(1)|\leq\frac{\Delta_i{b}\delta\varepsilon}{\sqrt{\Delta_i{C}}}
                        \right)
                \right]
                -\sum_\Phi\mathbb{E}\left[
                        \prod_{i=1}^m \phi_i
                    \right]
\end{align*}
as desired.
\end{proof}


Now the next three lemmas give the necessary estimates on the terms appearing in Lemma
\ref{p.kuelbsLiProp2}.

\begin{lemma}\label{l.kueblsLi4.2}
Let $f_i$, $g_i$, and $\Phi$ be as in Lemma \ref{p.kuelbsLiProp2}.  Then for any $\phi\in\Phi$
\begin{multline*}
  \limsup_{\varepsilon\downarrow0}
\varepsilon^{2\alpha/(1+\alpha)}|\log\varepsilon|^{\beta/(1+\alpha)}
		\log\mathbb{E}
    \left[
        \prod_{i=1}^m \phi_i(\varepsilon)
    \right] \\
        \le -2^{-\beta/(1+\alpha)}(1+\alpha)^{1+\beta/(1+\alpha)}
		\left(\frac{\pi^2}{8\alpha}\right)^{\alpha/(1+\alpha)}
		\sum_{i=1}^m \left(\frac{K(t_{i-1},t_i)}{d_i^\phi(\delta)^{2\alpha}}\right)^{1/(1+\alpha)}
\end{multline*}
where $ d_i^\phi(\delta) := \left\{\begin{array}{ll} b_i(1-\delta) & \text{if
}\phi_i=f_i \\ a_i(1+\delta) & \text{if } \phi_i=g_i\end{array}\right.$.
\end{lemma}

\begin{proof}
By the upper bound in \eqref{e.bmb},
\begin{equation*}
    \begin{aligned}
      \prod_{i=1}^m \phi_i
        &=  \prod_{i=1}^m P_C\left(
                \sup_{0\leq s\leq 1}|B(s)|
                \leq\frac{d_i^\phi(\delta)\varepsilon}{\sqrt{\Delta_i{C}}}
            \right) \\
        &\le  \left(
                \frac{4}{\pi}
            \right)^m
\exp\left(-\frac{\pi^2}{8\varepsilon^2}\sum_{i=1}^m\frac{\Delta_i{C}}{d_i^\phi(\delta)^2}\right).
    \end{aligned}
\end{equation*}
Then applying (\ref{e.EC}) completes the proof.
\end{proof}


\begin{lemma}\label{l.KuelbsLiLemma4.3}
Suppose that $\{\eta_i\}_{i=1}^m$ are nonnegative random variables such that, for any
$\beta_1,\ldots,\beta_m>0$, there exists $\alpha>0$, $\beta\in\mathbb{R}$, and $K>0$ such that
    \begin{equation*}
        \lim_{\varepsilon\downarrow0}\varepsilon^\alpha|\log\varepsilon|^\beta
            \log \P\left(\sum_{i=1}^m \beta_i\eta_i\leq\varepsilon\right)
        \ge - K.
    \end{equation*}
Let $P_\eta = \P(\cdot\mid \eta_1,\ldots,\eta_m)$.
Then, if $G$ is a standard normal random variable and
$\gamma_1,\ldots,\gamma_m>0$, we have that
    \begin{multline*}
        \liminf_{\lambda\rightarrow\infty}
		\lambda^{-\alpha/(1+\alpha)}(\log\lambda)^{\beta /(1+\alpha)}
		\mathbb{E}
            \left[
                \exp\left(-\lambda\sum_{i=1}^m \beta_i\eta_i\right)
                    \prod_{i=1}^m
	P_\eta\left(|G|\leq\frac{\gamma_i}{\sqrt{\lambda\eta_i}}\right)
            \right] \\
            \geq
                -(1+\alpha)^{1+\beta/(1+\alpha)}
			\alpha^{-\alpha/(1+\alpha)}K^{1/(1+\alpha)}.
    \end{multline*}
\end{lemma}

\begin{proof}
For any $L>0$, when $\sum_{i=1}^m \beta_i\eta_i \le L$, the
positivity of all parameters implies that $\eta_i\le L/\beta_i$ for each $i$ and thus
\[ \min_{1\le i\le m} \frac{\gamma_i}{\sqrt{\eta_i}}
		\ge \min_{1\le i\le m} \gamma_i\sqrt{\frac{\beta_i}{L}} > 0. \]
Also, note that for all sufficiently small $x>0$, one may choose $K'>0$ such that
$\P(|G|\le x) \ge K'x$.
Thus, for sufficiently large $\lambda$, there exists $K''>0$ such that
\begin{align}
\mathbb{E}&\notag
		\left[\exp\left(-\lambda\sum_{i=1}^m \beta_i\eta_i\right)
		\prod_{i=1}^m P_\eta\left(|G|\le
		\frac{\gamma_i}{\sqrt{\lambda\eta_i}}\right)\right] \\
	&\notag
	\ge \mathbb{E}\left[\exp\left(-\lambda\sum_{i=1}^m \beta_i\eta_i\right)
	\left(\min_{1\le i\le m} P_\eta\left(|G|\le
			\frac{\gamma_i}{\sqrt{\lambda\eta_i}}\right)\right)^m;
		\sum_{i=1}^m \beta_i\eta_i \le L\right]  \\
	&\label{e.11}
	\ge \left(\frac{K''}{\sqrt{\lambda}}\right)^m
		\mathbb{E}\left[\exp\left(-\lambda\sum_{i=1}^m \beta_i\eta_i\right);
		\sum_{i=1}^m \beta_i\eta_i \le L\right].
\end{align}
Thus, for any $\xi>0$, we may take $\theta(\lambda)
=\xi\lambda^{-1/(1+\alpha)}(\log\lambda)^{-\beta/(1+\alpha)}$, and
we have
\begin{align*}
&\liminf_{\lambda\rightarrow\infty}
	\lambda^{-\alpha/(1+\alpha)}(\log\lambda)^{\beta/(1+\alpha)}
		\log \mathbb{E}\left[\exp\left(-\lambda\sum_{i=1}^m \beta_i\eta_i\right);
		\sum_{i=1}^m \beta_i\eta_i \le L\right] \\
	&\ge \liminf_{\lambda\rightarrow\infty}
		\lambda^{-\alpha/(1+\alpha)} (\log\lambda)^{\beta/(1+\alpha)}
		\log \mathbb{E}\left[\exp\left(-\lambda\sum_{i=1}^m \beta_i\eta_i\right);
		\sum_{i=1}^m \beta_i\eta_i \le \theta(\lambda) \right] \\
	&\ge \liminf_{\lambda\rightarrow\infty}
		\lambda^{-\alpha/(1+\alpha)} (\log\lambda)^{\beta/(1+\alpha)}
		\left(-\theta(\lambda)\lambda + \log\P\left(\sum_{i=1}^m \beta_i\eta_i\le
		\theta(\lambda)\right)\right) \\
	&= -\xi + \xi^{-\alpha}(1+\alpha)^\beta\liminf_{\lambda\rightarrow\infty}
		\theta(\lambda)^{\alpha} |\log\theta(\lambda)|^{\beta}
		\log\P\left(\sum_{i=1}^m \beta_i\eta_i\le
		\theta(\lambda)\right) \\
	&\ge -\xi - \xi^{-\alpha}(1+\alpha)^\beta K
\end{align*}
In particular, combining this inequality with (\ref{e.11}) and
taking $\xi=(K\alpha(1+\alpha)^{\beta})^{1/(1+\alpha)}$ completes the
proof.
\end{proof}

\begin{lemma}\label{l.kuelbsLi4.4}
Let $f_i$ be as in Lemma \ref{p.kuelbsLiProp2}. Then
\begin{multline*}
    \liminf_{\varepsilon\downarrow0} \varepsilon^{2\alpha/(1+\alpha)}
		|\log\varepsilon|^{\beta/(1+\alpha)} \log \mathbb{E}
        \left[
            \prod_{i=1}^m f_i(\varepsilon)
            P_C
                \left(
                    |B(1)|\leq\frac{\Delta_i{b}\delta\varepsilon}{\sqrt{\Delta_i{C}}}
                \right)
        \right] \\
    \geq -2^{-\beta/(1+\alpha)}(1+\alpha)^{1+\beta/(1+\alpha)}
		\left(\frac{\pi^2}{8\alpha}\right)^{\alpha/(1+\alpha)}
		\sum_{i=1}^m
		\left(\frac{K(t_{i-1},t_i)}{b_i^{2\alpha}(1-\delta)^{2\alpha}}\right)^{1/(1+\alpha)}.
\end{multline*}
\end{lemma}

\begin{proof}
The lower bound in (\ref{e.bmb}) implies that
\begin{equation*}
    \begin{aligned}
      \prod_{i=1}^m f_i
        &=  \prod_{i=1}^m P_C\left(
                \sup_{0\leq s\leq1}|B(s)|
                \leq\frac{b_i(1-\delta)\varepsilon}{\sqrt{\Delta_i{C}}}
            \right) \\
        &\ge  \left(
                \frac{2}{\pi}
            \right)^m
\exp\left(-\frac{\pi^2}{8\varepsilon^2}\sum_{i=1}^m\frac{\Delta_i{C}}{b_i^2(1-\delta)^2}\right).
    \end{aligned}
\end{equation*}
Using this estimate to bound the desired expectation and applying
Lemma \ref{l.KuelbsLiLemma4.3} and equation (\ref{e.EC}) completes the proof.
\end{proof}


\begin{prop}\label{p.KuelbsLiProp3}
Under the hypotheses of Theorem \ref{t.SB}, we have that
  \begin{multline*}
    \liminf_{\varepsilon\downarrow 0}\varepsilon^{2\alpha/(1+\alpha)}
|\log\varepsilon|^{\beta/(1+\alpha)}
        \log \P\left(\bigcap_{i=1}^m
                \{a_i\varepsilon\leq M(t_i)\leq b_i\varepsilon\}
            \right) \\
        \geq
            -2^{-\beta/(1+\alpha)}(1+\alpha)^{1+\beta/(1+\alpha)}
		\left(\frac{\pi^2}{8\alpha}\right)^{\alpha/(1+\alpha)}
		\sum_{i=1}^m
			\left(\frac{K(t_{i-1},t_i)}{b_i^{2\alpha}}\right)^{1/(1+\alpha)}.
  \end{multline*}
\end{prop}

\begin{proof}
Clearly, for any $\gamma>0$,
\begin{align*}
 \P\left(\bigcap_{i=1}^m
                \{a_i\varepsilon\leq M(t_i)\leq b_i\varepsilon\}
            \right)
	&\ge \P\left(\bigcap_{i=1}^m
                \{a_i\varepsilon\leq M(t_i)\leq b_i\varepsilon\},
            |Z(t_m)|\le b_m\gamma\varepsilon\right).
\end{align*}
Thus, by Lemma \ref{p.kuelbsLiProp2}, for any $0<\delta<\gamma$ with $\delta$
sufficiently small that $a_i(1+\delta)<b_i(1-\delta)$ for each $i$, we have
that
\begin{align*}
 \P\left(\bigcap_{i=1}^m
                \{a_i\varepsilon\leq M(t_i)\leq b_i\varepsilon\}
            \right)	&\ge \mathbb{E}\left[
                    \prod_{i=1}^m f_iP_C\left(
                            |B(1)|\leq\frac{\Delta_i{b}\delta\varepsilon}{\sqrt{\Delta_i{C}}}
                        \right)
                \right]
                -\sum_{\phi\in\Phi}\mathbb{E}\left[
                        \prod_{i=1}^m \phi_i
                    \right].
\end{align*}
Now, given any $\phi\in\Phi$, the associated sequence
$\{d_i^\phi(\delta)\}_{i=1}^m$ (as defined
in Lemma \ref{l.kueblsLi4.2}) must satisfy
$d_i^\phi(\delta)=a_i(1+\delta)$ for at least one $i$.
Thus, for any $\phi\in\Phi$ we have that
\[ \sum_{i=1}^m \frac{K(t_{i-1},t_i)}{b_i(1-\delta)}
	< \sum_{i=1}^m \frac{K(t_{i-1},t_i)}{d_i^\phi(\delta)}. \]
Given this strict inequality, Lemmas \ref{l.kueblsLi4.2} and \ref{l.kuelbsLi4.4} imply that, for each
$\phi\in\Phi$
\[ \frac{\mathbb{E}\left[\prod_{i=1}^m \phi_i(\varepsilon)\right]}
	{\mathbb{E}\left[
                    \prod_{i=1}^m f_i(\varepsilon)P_C\left(
                            |B(1)|\leq\frac{\Delta_i{b}\delta\varepsilon}{\sqrt{\Delta_i{C}}}
                        \right)
                \right]} \rightarrow 0 \]
as $\varepsilon\downarrow0$.  This fact, combined with the identity $\log
(A-B) = \log A + \log(1-B/A)$ and again applying Lemma \ref{l.kuelbsLi4.4}
gives the desired result with $b_i$ replaced by $b_i(1-\delta)$.  Since
$\delta>0$ was arbitrary, allowing $\delta\downarrow0$ completes the proof.
\end{proof}

As alluded to in the discussion from Section \ref{s.intro}, a brief
review of the proof shows that conditioning  easily
determines the first-order ($m=1$) asymptotics of $Z=X\circ C$ for
other self-similar processes $X$ satisfying their own small ball estimates.
The following statement could also be inferred from the proofs of
\cite{AurzadaLifshits2009}.

\begin{prop}
\label{p.au}
Suppose that $\{\hat{C}(t)\}_{t\ge0}$ is continuous non-negative non-decreasing
and $\{X(t)\}_{t\ge0}$ is an $H$-self-similar process (that is, $\{X(ct)\}_{t\ge0}\overset{d}{=}\{c^HX(t)\}_{t\ge0}$ for any $c>0$)
which is independent of $\hat{C}$.  If there exist
$\alpha, \theta,\kappa>0$ and $K:(0,\infty)\rightarrow(0,\infty)$ such that
\[ \lim_{\varepsilon\downarrow0} \varepsilon^\alpha \log
	\P\left(
	\hat{C}(t)\le\varepsilon\right) = -K(t) \]
for all $t>0$ and
\begin{equation*}
\lim_{\varepsilon\downarrow0} \varepsilon^{\theta} \log \P\left(\sup_{s\in[0,1]}
	|X(s)|\le\varepsilon\right) = -\kappa,
\end{equation*}
then
\[
\lim_{\varepsilon\downarrow0} \varepsilon^{\alpha\theta/(\rho+\alpha)}
		\log \P\left(\sup_{s\in[0,t]} |X(\hat{C}(s))|\le\varepsilon\right) \\
	= - (\rho+\alpha)
		(\kappa^{\alpha}\rho^{-\rho}\alpha^{-\alpha}K(t)^\rho)^{1/(\rho+\alpha)}
\]
where $\rho=\theta H$.
\end{prop}

\begin{proof}
Under the assumptions on $X$, for any $\delta>0$, there exists
$\varepsilon_0=\varepsilon_0(\delta)$
such that for any $\varepsilon\in(0,\varepsilon_0)$
\[ \exp\left(-(1+\delta)\kappa \varepsilon^{-\theta} \right) 
	\le \P\left(\sup_{s\in[0,1]} |X(s)|\le\varepsilon\right)
	\le  \exp\left(-(1-\delta)\kappa \varepsilon^{-\theta} \right). \] 
Thus, there exist $c_1,c_2\in(0,\infty)$ depending only on $\varepsilon_0$ so that,
for all $\varepsilon>0$,
\[ c_1\exp\left(-(1+\delta)\kappa \varepsilon^{-\theta} \right) 
	\le \P\left(\sup_{s\in[0,1]} |X(s)|\le\varepsilon\right)
	\le  c_2\exp\left(-(1-\delta)\kappa \varepsilon^{-\theta} \right). \] 
Then continuity of $\hat{C}$ and self-similarity of $X$ implies that
\begin{multline*}
P_{\hat{C}}\left(\sup_{s\in[0,t]} |X(\hat{C}(s))|\le\varepsilon\right)
	= P_{\hat{C}}\left(\sup_{s\in[0,\hat{C}(t)]} |X(s)|\le\varepsilon\right) \\
	= P_{\hat{C}}\left(\sup_{s\in[0,1]} \hat{C}(t)^H|X(s)|\le\varepsilon\right)
	\le c_2\exp\left(-(1-\delta)\kappa \hat{C}(t)^\rho\varepsilon^{-\theta}
\right).
\end{multline*}
Taking expectations and applying the asymptotics of $\hat{C}$ gives
\begin{multline*}
\limsup_{\varepsilon\downarrow0} \varepsilon^{\alpha\theta/(\rho+\alpha)}
		\log \P\left(\sup_{s\in[0,t]} |X(\hat{C}(s))|\le\varepsilon\right) \\
	= - (\rho+\alpha)
		(((1-\delta)\kappa)^{\alpha}\rho^{-\rho}\alpha^{-\alpha}K(t)^\rho)^{1/(\rho+\alpha)}.
\end{multline*}
Letting $\delta\downarrow0$ proves the upper bound.  The lower bound
follows in a similar manner.
\end{proof}

\begin{remark}
Note that this result can be
more general than that for the two-sided diffusions in
\cite{AurzadaLifshits2009} where they require that $\theta H=1$. This equality is
often satisfied with the supremum norm, but there are
basic processes in this setting for which this does not hold.  For example, the process $C$
defined in (\ref{e.basic}) is 2-self-similar, but by Theorem \ref{t.6.4}
satisfies a small ball estimate with $\alpha=1$.  (And more generally, for
$\tilde{\rho}\equiv1$ and general $p\in[1,\infty)$, $\alpha=2/p$ and
$H=(p+2)/2$.)
\end{remark}

\begin{remark}
Note also that we could have again allowed a slowly varying factor in the
asymptotics of $\hat{C}$, but we have omitted it for ease.
\end{remark}

\begin{remark}
It is in the iterative arguments for Theorem \ref{t.SB} that
one uses, for example, the Gaussian properties of Brownian motion.  It is
clear that some of these estimates may be extended to other more general
processes.  For example, there is a known analogue of the Anderson inequality
that holds for symmetric $\alpha$-stable processes (see for example
Lemma 2.1 of \cite{ChenKuelbsLi2000}) that one could use to
extend the proof of Proposition \ref{p.kuelbsLiProp1}.

\end{remark}

\section{Applications to second order chaos}
\label{s.defZ}

Here we apply the results of the previous section
to prove small deviations estimates for
stochastic integrals of the form
\[
    Z_t=\int_0^t \omega(W_s,dW_s),
\]
where $W$ is an infinite-dimensional Brownian motion and $\omega$ is an
anti-symmetric continuous bilinear form.
Small deviations have been studied for analogous integrals of
finite-dimensional Brownian motions in \cite{Remillard1994} and
\cite{KuelbsLi2005}.

First we define the integral processes we study.  We will then
prove that these processes are equal in distribution to a
Brownian motion under an independent time-change,
and we establish a small ball estimate for the relevant
random clock.  Then by applying the results of Section
\ref{s.SB}, we are able to prove small deviations results for $Z$.
We fix the following notation for the sequel.

\begin{notation}
\label{n.BMW}
Let $(\mathcal{W},\mathcal{H},\mu)$ be a real abstract Wiener space (see for example \cite{Kuo1975} and \cite{Bogachev1998}).
We will let
\[ \mathcal{H}_*:= \{h\in \mathcal{H}: \langle h,\cdot\rangle \text{ extends
	to a continuous linear functional on } \mathcal{W}\}. \]
Let $\{W_t\}_{t\ge0}$ be a Brownian motion on $\mathcal{W}$ with variance
determined by
\[ \mathbb{E}[\langle W_s,h\rangle\langle W_s,k\rangle]
	= \langle h,k\rangle_H\min(s,t) \]
for all $s,t\ge0$ and $h,k\in\mathcal{H}_*$.  Let
$\omega:\mathcal{W}\times\mathcal{W}\rightarrow \mathbb{R}$ be a
anti-symmetric continuous bilinear map.
\end{notation}

\begin{remark}
\label{r.HS}
It is standard that continuity for a bilinear map $\omega$ on
$\mathcal{W}\times\mathcal{W}$ implies that the restriction of $\omega$ to
$\mathcal{H}\times\mathcal{H}$ is Hilbert-Schmidt, that is,
\[ \|\omega\|_{HS}^2
	:= \|\omega|_{\mathcal{H}\times\mathcal{H}}\|_{\mathcal{H}^{\otimes2}}^2
	:= \sum_{i,j=1}^\infty |\omega(h_i,h_j)|^2 <\infty \]
where $\{h_i\}_{i=1}^\infty$ is any orthonormal basis of $\mathcal{H}$; see for
example Proposition 3.14 of \cite{DriverGordina2008}.
\end{remark}

Associated to any abstract Wiener space is a class of canonical
projections.  Suppose that $P:\mathcal{H}\rightarrow \mathcal{H}$ is a finite-rank orthogonal
projection such that $P\mathcal{H}\subset \mathcal{H}_*$.  Let $\{ e_j\}_{j=1}^n$ be an orthonormal
basis for $P\mathcal{H}$.  Then we may extend $P$ to a (unique) continuous operator from
$\mathcal{W}\rightarrow \mathcal{H}$ (still denoted by $P$) by letting
\begin{equation}
\label{e.proj} Pw := \sum_{j=1}^n \langle w, e_j\rangle_\mathcal{H}  e_j
\end{equation}
for all $w\in \mathcal{W}$.

\begin{notation}\label{n.proj}
Let $\mathrm{Proj}(\mathcal{W})$ denote the collection of finite-rank
projections on $\mathcal{H}$
such that $P\mathcal{H}\subset \mathcal{H}_*$ and $P|_\mathcal{H}:\mathcal{H}\rightarrow \mathcal{H}$ is an orthogonal projection
(that is, $P$ has the form given in equation \eqref{e.proj}).
\end{notation}

For $P$ as in (\ref{e.proj}) and $\{W_t\}_{t\ge0}$ a Brownian motion on $\mathcal{W}$ as in Notation \ref{n.BMW}, $\{PW_t\}_{t\ge0}$ is a Brownian motion on the finite-dimensional space $\mathrm{Range}(P)$ and thus may be expressed as $PW_t=\sum_{j=1}^n W_t^je_j$ where the $W^j$'s are independent real-valued Brownian motions. We will let $\{Z_t^P\}_{t\ge0}$ denote the process defined by
\[
Z_t^P := \int_0^t \omega(PW_s,dPW_s).
\]
Note that, by the bilinearity and anti-symmetry of $\omega$, we may write
\begin{multline*}
Z_t^P = \int_0^t \omega\left(\sum_{j=1}^n W^j_se_j,\sum_{k=1}^n dW^k_se_k\right) \\
	= \sum_{j,k=1}^n \omega(e_j,e_k) \int_0^t W^j_s dW^k_s
	= \sum_{j<k} \omega(e_j,e_k) \int_0^t W^j_s dW^k_s-W^k_s dW^j_s;
\end{multline*}
thus, $\{Z_t^P\}_{t\ge0}$ is a continuous $L^2$-martingale.

It is well-known that $\mathcal{H}_*$ contains an
orthonormal basis of $\mathcal{H}$.  Thus, we may always take a sequence
$P_n\in\mathrm{Proj}(\mathcal{W})$ so that $P_n|_\mathcal{H}\uparrow
I_\mathcal{H}$.

\begin{prop}
\label{p.Mn}
If $\{P_n\}_{n=1}^\infty\subset\mathrm{Proj}(\mathcal{W})$ is an
sequence of projections such that $P_n|_\mathcal{H}\uparrow
I_\mathcal{H}$ and $Z_t^n:=Z_t^{P_n}$, then there exists an
$L^2$-martingale $\{Z_t\}_{t\ge0}$ such that, for all
$p\in[1,\infty)$ and $T>0$,
\begin{equation}\label{e.martingaleLimit}
     \lim_{n\rightarrow\infty} \mathbb{E}\left[\sup_{0\le t\le T}
    |Z_t^n-Z_t|^p\right]=0,
\end{equation}
and $\{Z_t\}_{t\ge0}$ is independent of the sequence of projections.
Thus, we will denote the limiting process by
\[
    Z_t = \int_0^t \omega(W_s,dW_s).
\]
The quadratic variation of $Z$ is given by
\begin{equation}
\label{e.qv}
\langle Z\rangle_t = \int_0^t \|\omega(W_s,\cdot)\|^2_{\mathcal{H}}\,ds
	:= \int_0^t \sum_{j=1}^\infty |\omega(W_s,e_j)|^2\,ds,
\end{equation}
where $\{e_j\}_{j=1}^\infty$ is an orthonormal basis of $\mathcal{H}$,
and, for all $p\in[1,\infty)$ and $T>0$, $\{Z_t\}_{t\ge0}$ satisfies
\begin{equation*}
\mathbb{E}\left[\sup_{0\le t\le T}
    |Z_t|^p\right]<\infty.
\end{equation*}
\end{prop}

\begin{proof}
First note that, for $P$ as in (\ref{e.proj}),
\begin{align*}
\mathbb{E}|Z_t^P|^2
	&= \mathbb{E}\langle Z^P\rangle_t
	= \sum_{j=1}^n \int_0^t \mathbb{E}\left|
		\omega( PW_s,e_j)\right|^2 ds \\
	&= \sum_{j,k=1}^n \int_0^t\int_0^{s_1}
		\left| \omega(e_k,e_j) \right|^2ds_2\,ds_1
	\le \frac{1}{2}t^2 \|\omega\|_{HS}^2 .
\end{align*}
Let $P,P'\in\mathrm{Proj}(\mathcal{W})$, and let $\{h_j\}_{j=1}^N$ be an
orthonormal basis for $P\mathcal{H}+P'\mathcal{H}$.  We then have that
\begin{align}
&\notag\mathbb{E}\left|Z_t^P-Z_t^{P'}\right|^2
	= \mathbb{E}\left|\int_0^t \left(\omega(PW_s,dPW_s)
		- \omega( P'W_s,dP'W_s)\right)\right|^2\\
	&\notag= \mathbb{E}\left|\int_0^t \left(\omega((P-P')W_s,dPW_s)
		+ \omega(P'W_s,d(P-P')W_s)\right)\right|^2\\
	&\notag\le 2 \mathbb{E}\left[\left|\int_0^t \omega((P-P')W_s,dPW_s)
		\right|^2 + \left|\int_0^t
		\omega(P'W_s,d(P-P')W_s)\right|^2\right] \\
	&\notag= t^2 \sum_{j,k=1}^N \left(
		\left| \omega((P-P')h_k,Ph_j) \right|^2
			+
		\left| \omega(P'h_k,(P-P')h_j)
		\right|^2\right)  \\
	&\label{e.aa}
	= t^2 \sum_{j,k=1}^\infty \left(
		\left| \omega((P-P')e_k,Pe_j) \right|^2
			+
		\left| \omega(P'e_k,(P-P')e_j)
		\right|^2\right).
\end{align}
Taking $P=P_n$ and $P'=P_m$ for $m\le n$ gives
\begin{align*}
\mathbb{E}\left[|Z_t^n-Z_t^{m}|^2\right]
	\le t^2 \left(\sum_{j=1}^n \sum_{k=m+1}^n |
			\omega( e_k,e_j)|^2
		+ \sum_{j=m+1}^n \sum_{k=1}^m |			
		\omega(e_k,e_j)|^2\right)
	\rightarrow 0
\end{align*}
as $m,n\rightarrow\infty$ since $\sum_{j,k=1}^\infty |	\omega(e_k,e_j)|^2=\|\omega\|_{HS}^2 <\infty$.
Since the space of continuous $L^2$-martingales on $[0,T]$ is complete in the
norm $N\mapsto\mathbb{E}|N_T|^2$, and, by Doob's maximal inequality, there
exists $c<\infty$ such that
\[ \mathbb{E}\left[\sup_{0\le t\le T} |N_t|^p \right] \le c \mathbb{E}|N_T|^p,
\]
it follows that there exists an $L^2$-martingale $\{Z_t\}_{t\ge0}$ such that
(\ref{e.martingaleLimit}) holds with $p=2$.  For $p>2$, since $Z$ is a chaos
expansion of order 2, a theorem of Nelson (see Lemma 2 of \cite{Nelson73b} and
pp. 216-217 of \cite{Nelson73c}) implies that, for each $j\in\mathbb{N}$, there
exists $c_j<\infty$ such that
\[ \mathbb{E}|Z_t^n-Z_t|^{2j} \le c_j\left(\mathbb{E}|Z_t^n-Z_t|^2\right)^j, \]
and again this combined with Doob's maximal inequality is sufficient to prove
(\ref{e.martingaleLimit}).

One may similarly use (\ref{e.aa}) to show that, for
$\{e_j'\}_{j=1}^\infty\subset \mathcal{H}_*$ another orthonormal basis of
$\mathcal{H}$ and $P_n'$ a corresponding sequence of orthogonal projections,
that
\[ \lim_{n\rightarrow\infty}
	\mathbb{E}\left[\sup_{0\le t\le T} \left|Z_t^{P_n} - Z_t^{P_n'}\right|^p \right] =0 \]
and thus $Z$ is independent of choice of basis.

Since the quadratic variation of $Z^n$ is given by
\[ \langle Z^n \rangle_t
	= \int_0^t |\omega(P_nB_s,dP_nB_s)|^2
	= \int_0^t \sum_{j=1}^n |\omega(P_nB_s,e_j)|^2\,ds \]
and
\begin{align*}
\mathbb{E}|\langle Z\rangle_t-\langle Z^n\rangle_t|
	&\le \sqrt{\mathbb{E}|\langle Z-Z^n\rangle_t| \cdot
		\mathbb{E}|\langle Z+Z^n\rangle_t|} \\
	&= \sqrt{\mathbb{E}| Z_t-Z^n_t|^2 \cdot
		\mathbb{E}|Z_t+Z^n_t|}
	\rightarrow0
\end{align*}
as $n\rightarrow\infty$ and (\ref{e.qv}) follows.

\end{proof}

More general integrals of the form above are considered in
\cite{DriverGordina2008}, in the context of Brownian motions on certain
infinite-dimensional Lie groups, and the above proposition is a special case of
Proposition 4.1 of that reference. In particular, processes like the $Z_t$ defined in Proposition \ref{p.Mn} appear as the central component of hypoelliptic Brownian motions on infinite-dimensional Heisenberg groups with one-dimensional center, and more generally as a term in Brownian motions on infinite-dimensional nilpotent Lie groups.

We give the following basic example of the type of process $Z$ we study here.

\begin{example}\label{ex.workingZt}
Let $q=\{q_j\}_{j=1}^\infty\in\ell^1(\mathbb{R}^+)$ and set
\[
    \mathcal{W} := \ell^2_q(\mathbb{C}):=\left\{v\in \mathbb{C}^\mathbb{N} :
		\sum_{j=1}^\infty q_j|v_j|^2 < \infty\right\}
\]
and
$
    \mathcal{H}=\ell^2(\mathbb{C})$
where both $\mathcal{W}$ and $\mathcal{H}$ are considered as vector spaces
over $\mathbb{R}.$  Then $(\mathcal{W},\mathcal{H})$ determines an abstract
Wiener space (for example, Example 3.9.7 of \cite{Bogachev1998}).
Define $\omega:\mathcal{W}\times\mathcal{W}\rightarrow \mathbb{R}$ by
\[
    \omega(w,w') = \sum_{j=1}^\infty q_j\mathrm{Im}(\bar{w}_jw_j')
	= \sum_{j=1}^\infty q_j(x_jy_j'-y_jx_j')
\]
where $w_j=x_j+iy_j$ for each $j$.  Then for a Brownian motion $W=\{X^j+iY^j\}_{j=1}^\infty$,
where $\{X^j,Y^j\}_{j=1}^\infty$ are independent
standard real-valued Brownian motions, we have
that
\begin{align*}
  Z(t)   &=  \int_0^t \omega\left(
                                W_s,dW_s
                            \right)
	=  \sum_{j=1}^\infty q_j \int_0^t    X_s^jdY_s^j-Y_s^jdX_s^j
\end{align*}
is an infinite weighted sum of independent L\'evy areas. (Note that, since the weights $\{q_j\}$ are $\ell^1$, this expression for $Z$ makes sense. Indeed, in order for the Brownian motion $W$ to make sense on $\mathcal{W}$, these weights must be $\ell^1$. See \cite{Bogachev1998} for more details.)
\end{example}

\begin{remark}
\label{r.qv}
Since $Z$ is a martingale with
\begin{align*}
\langle Z\rangle_t
	&= \int_0^t \|\omega(W_s,\cdot)\|_{\mathcal{H}}^2 \,ds
	= \int_0^t \sum_{j=1}^\infty |\omega(W_s,e_j)|^2\,ds \\
	&= \int_0^t \sum_{j=1}^\infty\sum_{k=1}^\infty
		|W_s^k\omega(e_k,e_j)|^2\,ds
	= \sum_{j=1}^\infty\sum_{k=1}^\infty
		|\omega(e_k,e_j)|^2 \int_0^t  (W_s^k)^2 \,ds \\
	&= \sum_{k=1}^\infty
		\|\omega(e_k,\cdot)\|_{\mathcal{H}}^2 \int_0^t  (W_s^k)^2 \,ds,
\end{align*}
we know there exists a (not necessarily independent) real-valued Brownian motion $B$
such that $Z(t)=B(\langle Z\rangle_t)$ by the Dubins-Schwarz representation (see for example Theorem 34.1 on page 64 of \cite{rogers2000diffusions}).
We will show in the next section that this representation in fact holds with
$B$ an independent Brownian motion.
\end{remark}

\subsection{A representation theorem}
\label{ss.rep}
In this section, we show that $Z\overset{d}{=}B(\langle Z\rangle)$ for an independent
Brownian motion $B$.  This representation is well-known for $Z$ the standard
stochastic L\'evy area for two-dimensional Brownian motion (see for example
Example 6.1 on page 470 of \cite{ikeda1989stochastic}), and was
also proved for more general stochastic integrals of finite-dimensional
Brownian motions in \cite{KuelbsLi2005}.  We summarize the latter result now;
see Section 3 of \cite{KuelbsLi2005} for a proof.

\begin{lemma}\label{l.kuelbsLi3.3}
Let $W$ be a standard Brownian motion in $\mathbb{R}^n$ and
$A$ be a real non-zero skew-symmetric $n\times n$ matrix with
non-zero eigenvalues $\{\pm ia_j\}_{j=1}^r$ (where $2r\le n$ and 0 is an eigenvalue of multiplicity $n-2r$).
For $t>0$, let
\[ L(t) := \int_0^t \langle AW_s,dW_s\rangle \]
and
  \[
    \tilde{L}(t):=B
        \left(
            \sum_{j=1}^r a_j^2\int_0^t (X_s^j)^2+(Y_s^j)^2\,ds
        \right),
    \]
where $B$ and $\{X_j,Y_j\}_{j=1}^{r}$ are independent standard real-valued
Brownian motions.
Then the law of $\{L(t)\}_{t\geq0}$ is equal to the
law of $\{\tilde{L}(t)\}_{t\geq0}$.
\end{lemma}

\begin{remark}
\label{r.Q}
In particular, this lemma implies that each of the finite-dimensional approximations
$Z^n$ to $Z$ has such a representation, in the following way.
By Remark \ref{r.HS}, the continuity assumption for $\omega$ implies that
its restriction to the Cameron-Martin space is Hilbert-Schmidt, and thus the
Riesz representation theorem implies the existence of an anti-symmetric
Hilbert-Schmidt operator $Q=Q_\omega:\mathcal{H}\rightarrow\mathcal{H}$ such that
\[ \omega(h,k) = \langle Q h,k\rangle_\mathcal{H}, \quad \text{ for all }
h,k\in\mathcal{H}. \]
Thus,
\begin{align*}
Z_t^P
	= \int_0^t \omega(PB_s,dPB_s)
	= \int_0^t \langle QPB_s,dPB_s\rangle_\mathcal{H}
	= \int_0^t \langle (PQP)PB_s,dPB_s\rangle_\mathcal{H},
\end{align*}
and we may apply Lemma \ref{l.kuelbsLi3.3} to $Z^P$,
as $PB$ is a Brownian motion on the finite-dimensional space
$P\mathcal{H}\subset \mathcal{H}$
and $A=PQP$ is a skew-symmetric linear operator on $P\mathcal{H}$.
\end{remark}

We will use this representation for the finite-dimensional approximations to show that an
analogous statement is true for $Z$.  First we record the following simple
lemma.

\begin{lemma}
\label{l.HS}
 Let $Q:\mathcal{H}\rightarrow\mathcal{H}$ be a Hilbert-Schmidt operator,
and let $P_n$ be an increasing sequence of orthogonal projections on $\mathcal{H}$ such that
$P_n|_\mathcal{H}\uparrow I_\mathcal{H}$.  Then, as $n \to \infty$,
$P_n Q P_n \to Q$ in Hilbert-Schmidt norm.
\end{lemma}

\begin{proof}
Let $\{e_i\}_{i=1}^\infty$ be an orthonormal basis of $\mathcal{H}$ so that
$\{e_i\}_{i=1}^{r_n}$ is an orthonormal basis of $P_n\mathcal{H}$.  We have
  \begin{align*}
\|P_n Q P_n - Q\|_{HS}^2
	&= \sum_{i=1}^\infty \|(P_n Q P_n -
    		Q)e_i\|_\mathcal{H}^2 \\
    &= \sum_{i=1}^{r_n} \|(P_n - I) Q e_i\|_H^2 +
    \sum_{i=r_n+1}^\infty \|Q e_i\|_\mathcal{H}^2 \\
    &\le \sum_{i=1}^{\infty} \|(P_n - I) Q e_i\|_\mathcal{H}^2 +
    \sum_{i=r_n+1}^\infty \|Q e_i\|_\mathcal{H}^2.
  \end{align*}
  The second term goes to zero since it is the tail of the convergent
  sum $\sum_{i=1}^\infty \|Q e_i\|_\mathcal{H}^2 = \|Q\|_{HS}^2 < \infty$.  For
  the first term, we may use the dominated convergence theorem: since
  $P_n \to I$ strongly we have $\|(P_n - I) Q e_i\|_\mathcal{H}^2 \to 0$ for
  each $i$, and
	$\|(P_n - I) Q
  e_i\|_\mathcal{H}^2 \le 4 \|Q e_i\|_\mathcal{H}^2$ which is summable.
\end{proof}

Now we may prove the desired representation for $Z$.

\begin{theorem}\label{t.Representation}
Let  $Z(t) = \int_0^t \omega(W_s,dW_s)$ be
as defined in Proposition \ref{p.Mn}, and let $Q=Q_\omega$ be the linear
operator on $\mathcal{H}$ such that $\omega(h,k)=\langle
Qh,k\rangle_\mathcal{H}$ for all $h,k\in\mathcal{H}$ as in Remark \ref{r.Q}.
Let $\{X^j,Y^j\}_{j=1}^\infty$ be independent
standard real-valued Brownian motions, $\{\pm iq_j\}_{j=1}^\infty$ be the
eigenvalues of  $Q$ so that $\{q_j\}_{j=1}^\infty$ is ordered from largest to smallest,
and define for $t\ge0$
\[
    C(t):=\sum_{j=1}^\infty q_j^2\int_0^t (X_s^j)^2+(Y_s^j)^2\,ds.
\]
(Note that $C(t)$ is well-defined and finite almost surely for each $t$.)
Then the
  law of $\{Z(t)\}_{t\geq0}$ is equal to the law of $\{\tilde{Z}(t)\}_{t\geq0}$ where
  $\tilde{Z}(t)= B(C(t))$ for $B$ a standard Brownian motion independent of
$\{X^j,Y^j\}_{j=1}^\infty$.
\end{theorem}

\begin{proof}
    Let $\{P_n\}_{n=1}^\infty\subset\mathrm{Proj}(W)$ be such that
$P_n|_\mathcal{H}\uparrow I_\mathcal{H}$ and
    \[
        Z^n(t)   = \int_0^t \omega(P_nW_s,dP_nW_s)
	= \int_0^t \langle (P_nQP_n)P_nW_s,dP_nW_s\rangle_\mathcal{H},
    \]
as in Proposition \ref{p.Mn},
Then Lemma \ref{l.kuelbsLi3.3} implies that, for each  $n$,
the law of $\{Z^n(t)\}_{t\geq0}$ is equal to
    the law of $\{\tilde{Z}^n(t)\}_{t\geq0}$ where
    \[
        \tilde{Z}^n(t):= B(C_t^n)
		:= B\left(
            \sum_{j=1}^{r_n}
                q_{nj}^2 \int_0^t(X_s^j)^2+(Y_s^j)^2\ ds
        \right)
    \]
where $\{\pm iq_{nj}\}_{j=1}^{r_n}$ are the non-zero eigenvalues of
$P_nQP_n$.  For each $n$, we will assume the $q_{nj}$ are ordered in $j$
from largest to smallest.
Clearly, Proposition \ref{p.Mn} and in particular
(\ref{e.martingaleLimit})
imply that
$Z^n\Rightarrow Z$ and the collection $\{Z^n\}_{n=0}^\infty$
is tight.  Equality in distribution then implies that
$\tilde{Z}^n\Rightarrow Z$ and $\{\tilde{Z}^n\}_{n\geq0}$
is tight.

Now we also have that, for each fixed $t>0$,
\begin{align*}
\mathbb{E}|&\tilde{Z}(t)-\tilde{Z}^n(t)|^2
	= \mathbb{E}[\mathbb{E}[|\tilde{Z}(t)-\tilde{Z}^n(t)|^2 |C,C^n]]
	= \mathbb{E}|C(t)-C^n(t)| \\
	&= \mathbb{E}\left|\sum_{j=1}^\infty q_j^2\int_0^t (X^j_s)^2+(Y^j_s)^2\,ds
		- \sum_{j=1}^{r_n} q_{nj}^2\int_0^t (X^j_s)^2+(Y^j_s)^2\,ds\right|
		\\
	&= \mathbb{E}\left|\sum_{j=1}^{r_n} (q_j^2-q_{nj}^2)\int_0^t (X^j_s)^2+(Y^j_s)^2\,ds
		- \sum_{j=r_n+1}^\infty q_j^2\int_0^t (X^j_s)^2+(Y^j_s)^2\,ds\right|
		\\
	&\le \mathbb{E}\left|\sum_{j=1}^{r_n} (q_j^2-q_{nj}^2)\int_0^t
			(X^j_s)^2+(Y^j_s)^2\,ds \right|
		+ \mathbb{E}\left| \sum_{j=r_n+1}^\infty q_j^2\int_0^t
		(X^j_s)^2+(Y^j_s)^2\,ds\right|.
\end{align*}
Note that
\[ \|Q_n\|_{HS}^2 = 2\sum_{j=1}^{r_n} q_{nj}^2  \quad\text{and}\quad
	\|Q\|_{HS}^2 = 2\sum_{j=1}^\infty q_j^2<\infty. \]
Thus, for the second term,
\begin{align*}
\mathbb{E}\left| \sum_{j=r_n+1}^\infty q_j^2\int_0^t (X^j_s)^2+(Y^j_s)^2\,ds\right|
	&= t^2 \sum_{j=r_n+1}^\infty q_j^2 \rightarrow 0,
\end{align*}
clearly, since this is the tail of a convergent sequence. For the first term,
\begin{align*}
\mathbb{E}\left|\sum_{j=1}^{r_n} (q_j^2-q_{nj}^2)\int_0^t
			(X^j_s)^2+(Y^j_s)^2\,ds \right|
	&\le t^2\sum_{j=1}^{r_n} |q_j^2-q_{nj}^2|
	= t^2\sum_{j=1}^{r_n} (q_j^2-q_{nj}^2) \\
	&\le t^2\left(\sum_{j=1}^\infty q_j^2 - \sum_{j=1}^{r_n} q_{nj}^2\right)
	\rightarrow 0,
\end{align*}
where the equality follows from the min-max theorem which implies that
\begin{align*}
q_{nj}
	&= \sup_{\tiny{\begin{array}{cc}S\subset H \\
		\mathrm{dim}(S)=j\end{array}}} \min_{h\in S} \frac{\|Q_nh\|}{\|h\|} \\
	&\le \sup_{\tiny{\begin{array}{cc}S\subset H \\
		\mathrm{dim}(S)=j\end{array}}} \min_{h\in P_nS} \frac{\|Qh\|}{\|h\|}
	= \sup_{\tiny{\begin{array}{cc}S\subset P_nH \\
		\mathrm{dim}(S)=j\end{array}}} \min_{h\in S} \frac{\|Qh\|}{\|h\|}
	\le q_j
\end{align*}
and the limit follows from Lemma \ref{l.HS} which implies that the
Hilbert-Schmidt norms of $Q_n$ converge to the Hilbert-Schmidt norm of $Q$.
Thus, for any $(t_1,\ldots,t_m)\in(\mathbb{R}^+)^m$,
    \[
      \E\left|(\tilde{Z}_{t_1}^n,\ldots,\tilde{Z}_{t_m}^n)
            -(\tilde{Z}_{t_1},\ldots,\tilde{Z}_{t_m})\right|^2
      = \sum_{i=1}^m
      \E|\tilde{Z}^n(t_i)-\tilde{Z}(t_i)|^2\rightarrow0
    \]
as $n\rightarrow\infty$, and
the finite-dimensional distributions of $\tilde{Z}^n$ converge to
those of $\tilde{Z}$. Combining this with the tightness of
$\{\tilde{Z}^n\}$ implies that $\tilde{Z}^n\Rightarrow\tilde{Z}$. However, since
$\tilde{Z}^n\Rightarrow Z$ also, it must be that $\{\tilde{Z}(t)\}_{t\geq0}$ and
    $\{Z(t)\}_{t\geq0}$ are equal in distribution.
\end{proof}

Noting that, for $\{e_k\}_{k=1}^\infty$ an orthonormal basis of $\mathcal{H}$,
\[ \sum_{k=1}^\infty \|\omega(e_k,\cdot)\|_{\mathcal{H}}^2
	= \sum_{k=1}^\infty \|\langle Qe_k,\cdot\rangle_{\mathcal{H}}\|_{\mathcal{H}}^2
	= \sum_{k=1}^\infty \|Qe_k\|_\mathcal{H}^2
	= 2 \sum_{k=1}^\infty q_k^2,
\]
we see that indeed $C(t)=\langle Z\rangle_t$ up to a reordering of terms.
Given this last theorem, in order to prove small deviations for $Z$, it suffices to
prove them for $\tilde{Z}$.  The results of Section \ref{s.SB} lead us to find
a small ball estimate for the process $\langle Z\rangle$.

\subsection{Small deviations for $\langle Z\rangle_t$ and applications}
\label{ss.smallBall}

Note again that, as in Remark \ref{r.qv} we may write $ \langle Z\rangle_t = \sum_{k=1}^\infty
\|\omega(e_k,\cdot)\|_{\mathcal{H}}^2\xi_k(t)$ where
$\{\xi_k\}_{k=1}^\infty$ are i.i.d.~copies of
\begin{equation}
\label{e.111} \xi(t) := \int_0^t B_s^2 \,ds.
\end{equation}
Recall that, if $\{\zeta_j\}_{j=1}^m$ are
independent positive random variables satisfying small ball
estimates with the same exponents $\alpha$ and $\beta$ for coefficients
$\{K_j\}_{j=1}^m$, then
\begin{multline*}
\lim_{\lambda\rightarrow\infty}
	\lambda^{-\alpha/(1+\alpha)}(\log\lambda)^{\beta/(1+\alpha)}
	\log\E\left[e^{-\lambda \sum_{j=1}^m \zeta_j }\right] \\
	= -(1+\alpha)^{1+\beta/(1+\alpha)}\alpha^{-\alpha/(1+\alpha)}
\sum_{j=1}^m K_j^{1/(1+\alpha)}
\end{multline*}
and equivalently
\[ \lim_{\varepsilon\downarrow 0}
	\varepsilon^{\alpha} |\log\varepsilon|^{\beta} \log \P\left(\sum_{j=1}^m
		\zeta_j \le \varepsilon\right)
	= - \left(\sum_{j=1}^m K_j^{1/(1+\alpha)}\right)^{(1+\alpha)}. \]
In particular, if $\{\eta_j\}_{j=1}^m$ are positive i.i.d.~random variables satisfying small ball
estimates with $K_j=K$ for each
$j$ and $\zeta_j=a_j\eta_j$ for some $a_j>0$, then we have that
\[ \lim_{\varepsilon\downarrow 0}
	\varepsilon^{\alpha} |\log\varepsilon|^{\beta} \log \P\left(\sum_{j=1}^m
		a_j\eta_j \le \varepsilon\right)
	= - K \left(\sum_{j=1}^m a_j^{\alpha/(1+\alpha)}\right)^{(1+\alpha)}. \]
Equivalently,
\begin{multline}
\label{e.EC2}    \limsup_{\lambda\rightarrow\infty}
        \lambda^{-\alpha/(1+\alpha)}(\log\lambda)^{\beta/(1+\alpha)}\log\mathbb{E}
            \left[
                e^{-\lambda\sum_{j=1}^\infty a_j \zeta_j}
            \right] \\
	= -(1+\alpha)^{1+\beta/(1+\alpha)}\alpha^{-\alpha/(1+\alpha)}
		K^{1/(1+\alpha)} \sum_{j=1}^m a_j^{\alpha/(1+\alpha)}.
\end{multline}
In the event this sum is actually infinite with a summable sequence of
coefficients $\{a_j\}_{j=1}^\infty$, analogous results hold under  additional requirements on the
coefficients.  Small deviations of random variables of the form
\[
    S=\sum_{j=1}^\infty a_j\zeta_j
\]
where $\{a_j\}\in\ell^1(\mathbb{R}^+)$ 
and $\{\zeta_j\}$
are non-negative i.i.d.~random variables, have  been studied in
\cite{aurzada2008short,borovkov2008small,mayer1993probability,rozovsky2011small,rozovskii2007small,lifshits1997lower}.
In particular, we present the following theorem (Theorem 3.1 of \cite{borovkov2008small}) without proof.

\begin{theorem}\label{t.BR3.1}
Suppose that $\zeta$ is a non-negative random variable such that there exist
$\alpha>0$ and a slowly varying function $L$ so that
  \[
    \log\P(\zeta<\varepsilon)
\sim -\varepsilon^{-\alpha}L(\varepsilon)
  \]
as $\varepsilon\downarrow0$, and there exist $\kappa, \delta>0$ so that
\[
    \E\left[\zeta^{(1/(\gamma+\kappa))+\delta}\right]<\infty,
  \]
where $\gamma=\frac{1+\alpha}{\alpha}$.
Then, given a sequence $\{a_j\}_{j=1}^\infty\subset\mathbb{R}^+$ such that
$a_j=O(j^{-(\gamma+\kappa)})$ and $\{\zeta_j\}_{j=1}^\infty$
i.i.d.~copies of $\zeta$,
  \[
     \log\P\left(\sum_{j=1}^\infty
a_j\zeta_j<\varepsilon\right)
	\sim -\left(\sum_{j=1}^\infty
    a_j^{\alpha/(1+\alpha)}\right)^{(1+\alpha)}\varepsilon^{-\alpha}L(\varepsilon)
  \]
as $\varepsilon\downarrow0$.
\end{theorem}

\begin{remark}
\label{r.1}
For any $\{a_j\}\in\ell^1(\mathbb{R}^+)$, we easily obtain the upper bound for
(\ref{e.EC2})
in the following way. First, note that
    \[
        \E\left[e^{-\lambda\sum_{j=1}^\infty a_j\zeta_j}\right]
        =
        \prod_{j=1}^\infty \E\left[e^{-\lambda a_j\zeta_j}\right]
    \]
by independence and bounded convergence.    Thus,
    \begin{align*}
      \limsup_{\lambda\rightarrow\infty}
        \lambda^{-\alpha/(1+\alpha)}&
		(\log\lambda)^{\beta/(1+\alpha)}\log\E\left[e^{-\lambda\sum_{j=1}^\infty
a_j\zeta_j}\right] \\
      &\leq  \sum_{j=1}^\infty
        \limsup_{\lambda\rightarrow\infty}
            \lambda^{-\alpha/(1+\alpha)}(\log\lambda)^{\beta/(1+\alpha)}\log
                \E\left[e^{-\lambda a_j\zeta_j}\right] \\
      &= -(1+\alpha)^{1+\beta/(1+\alpha)}\alpha^{-\alpha/(1+\alpha)}
		K^{1/(1+\alpha)}\sum_{j=1}^\infty a_i^{\alpha/(1+\alpha) }
    \end{align*}
    by reverse Fatou's lemma for non-positive functions.
\end{remark}

\begin{remark}
\label{r.2}
  When $\{a_j\}\in\ell^1(\mathbb{R}^+)$ are such that $a_j\le \tilde{a}_j$ for
$\{\tilde{a}_j\}$
a sequence as in Theorem \ref{t.BR3.1},
then we may easily obtain a lower bound in terms of $\{\tilde{a}_j\}$.
  Since $a_j\leq \tilde{a}_j,$
  \[
    S:=\sum_{j=1}^\infty a_j\zeta_j
    \leq
    \sum_{j=1}^\infty \tilde{a}_j\zeta_j=:\tilde{S}
  \] and thus
  $\P(\tilde{S}\leq\varepsilon)\leq\P(S\leq\varepsilon).$
  It follows that
  \begin{align*}
    \liminf_{\varepsilon\downarrow0}
		\varepsilon^\alpha \log \P(S\leq\varepsilon)
        &\geq \liminf_{\varepsilon\downarrow0}
		\varepsilon^\alpha \log \P(\tilde{S}\leq\varepsilon)
        =  -K\left(\sum_{j=1}^\infty
        	\tilde{a}_j^{\alpha/(1+\alpha)} \right)^{(1+\alpha)}.
  \end{align*}
Similarly,
\begin{align*}
    \liminf_{\lambda\rightarrow\infty}
	\lambda^{-\alpha/(1+\alpha)} \log\E\left[e^{-\lambda
    S^q}\right]
	&\ge \liminf_{\lambda\rightarrow\infty}
		\lambda^{-\alpha/(1+\alpha)} \log\E\left[e^{-\lambda
	    S^a}\right] \\
		&= -(1+\alpha)\alpha^{-\alpha/(1+\alpha)}
			K\sum_{j=1}^\infty a_j^{\alpha/(1+\alpha)} .
\end{align*}
\end{remark}

\begin{prop}\label{p.kuelbsLiLemma4.1}
Let $Z(t)=\int_0^t \omega(W_s,dW_s)$ be as in Proposition \ref{p.Mn}, and
suppose that $\|\omega(e_j,\cdot)\|_{\mathcal{H}}=O(j^{-r})$ for
$r>1$.
Then $C=\langle Z\rangle$ satisfies
Assumption \ref{a.a} with $\alpha=1$, $\beta=0$, and
\[ K(t_{i-1},t_i) = \frac{1}{8} \|\omega\|_1^2 (\Delta_i t)^2 \] where $\Delta_i t:=t_i-t_{i-1}$
and
\[ \|\omega\|_1 := \sum_{j=1}^\infty
	\|\omega(e_j,\cdot)\|_{\mathcal{H}}<\infty. \]
That is, for any $m\in\mathbb{N}$,
$0=t_0<t_1<\cdots<t_m$, and $\{d_i\}_{i=1}^m$ a decreasing sequence,
\begin{equation*}
    \lim_{\varepsilon\downarrow 0}\varepsilon\log
    \P\left(
        \sum_{i=1}^m d_i \Delta_i \langle Z\rangle\leq\varepsilon
    \right)=
            -\frac{1}{8}\|\omega\|_1^2 \left(\sum_{i=1}^m d_i^{1/2}\Delta_i t\right)^2.
\end{equation*}
\end{prop}

\begin{proof}
By Equation (\ref{e.111}), we have that $\sum_{i=1}^m d_i \Delta_i C = \sum_{j=1}^\infty
\|\omega(e_j,\cdot)\|_{\mathcal{H}}^2\zeta_j$ where $\{\zeta_j\}_{j=1}^\infty$ are i.i.d.~copies of
\[ \zeta := \sum_{i=1}^m d_i \int_{t_{i-1}}^{t_i} (B_s)^2 \,ds. \]
Theorem \ref{t.6.4} implies that
\begin{equation*}
  \lim_{\varepsilon\downarrow0}
	\varepsilon \log
\P\left(\zeta\leq\varepsilon\right)
	=-\frac{1}{8}\left(\sum_{i=1}^m d_i^{1/2}\Delta_i t\right)^2.
\end{equation*}
Thus, under the assumptions on $\omega$, the desired result follows from
Theorem \ref{t.BR3.1}
with $\alpha=1$ and $a_j=\|\omega(e_j,\cdot)\|_\mathcal{H}^2$.
\end{proof}

Now combining this result with Theorem \ref{t.Representation}
and Theorem \ref{t.SB} with $\alpha=1$, $\beta=0$, and
$K(t_{i-1},t_i)=\frac{1}{2}\|\omega\|_1^2(\Delta_it)^2$ immediately yields the
following.

\begin{theorem}\label{t.smallBall}
Let $Z(t)=\int_0^t \omega(W_s,dW_s)$ with $\|\omega(e_j,\cdot)\|_{\mathcal{H}}=O(j^{-r})$ for $r>1$.
Then, for any $m\in\mathbb{N}$,
$0=t_0<t_1<\cdots<t_m$, and
$0\leq a_1<b_1\leq a_2<b_2\leq\cdots\leq a_m<b_m$,
    \begin{equation*}
      \lim_{\varepsilon\downarrow0}
            \varepsilon\log P\left( \bigcap_{i=1}^m
                               \{ a_i\varepsilon\leq \sup_{0\leq s\leq
t_i}|Z(s)|\leq b_i\varepsilon \}                            \right)
                        =
                            -\frac{\pi}{4}\|\omega\|_1\sum_{i=1}^m\frac{\Delta_i
t}{b_i}.
    \end{equation*}
\end{theorem}

\begin{remark}
Note that the logarithmic asymptotics determined by Remarks
\ref{r.1} and
\ref{r.2} are sufficient to show that Theorem \ref{t.SB} has the correct
order of asymptotics for general $\omega$ satisfying $\|\omega\|_1<\infty$.
\end{remark}


As was done in \cite{KuelbsLi2005}, Theorem \ref{t.smallBall} may be used to
prove a functional law of the iterated logarithm for $Z$.
This immediately implies a Chung-like law of the
iterated logarithm, or one may prove this directly from the first-order small
deviations estimates proved in Theorem \ref{t.smallBall} as was done in
\cite{Remillard1994}.  The proofs follow exactly analogously to
the finite-dimensional cases in \cite{Remillard1994}
and \cite{KuelbsLi2005}, so we omit the proofs here.

\begin{theorem}
\label{t.chungLIL}
Let $Z(t)=\int_0^t \omega(W_s,dW_s)$ with $\|\omega(e_j,\cdot)\|_{\mathcal{H}}=O(j^{-r})$ for
$r>1$.
Then
\[
    \P\left(
        \liminf_{t\rightarrow\infty} \frac{\log\log{t}}{t} \sup_{0\leq s\leq t}
        \left|Z(s)\right|= \frac{\pi}{4}\|\omega\|_1
    \right)=1.
\]
\end{theorem}

\begin{theorem}\label{t.FLIL}
Let $Z(t)=\int_0^t \omega(W_s,dW_s)$ with $\|\omega(e_j,\cdot)\|_{\mathcal{H}}=O(j^{-r})$ for
$r>1$. Let
\[
\eta_n(t):=\frac{\log\log{n}}{\frac{\pi}{4}\|\omega\|_1 n}\sup_{0\le s\le nt}
|Z(s)|, \]
and let $\mathcal{M}$ denote the set of non-negative, non-decreasing continuous
  functions such that $f(0)=0$ and $\displaystyle \lim_{t\rightarrow\infty}
  f(t)=\infty$.
  Then, with probability 1, $\{\eta_n\}$ is relatively compact in
$\mathcal{M}$ and  the set of cluster points of $\{\eta_n\}$ is
  \begin{equation*}
    \left\{
            f\in\mathcal{M}:\int_0^\infty f^{-1}(s)\,ds\leq 1
        \right\}.
  \end{equation*}
\end{theorem}

From here it is possible to obtain various occupation measure results for the maximal
process of $Z$, as was done in \cite{ChenKuelbsLi2000}, \cite{KuelbsLi2002}, and
\cite{KuelbsLi2005}.

\begin{remark}
Note that Theorems \ref{t.smallBall}, \ref{t.chungLIL}, and \ref{t.FLIL} also
include the finite-dimensional stochastic integrals already studied in
\cite{Remillard1994} and \cite{KuelbsLi2005}.  The difference in factors of 2
arises from the fact that the non-zero singular values of $Q$ necessarily have
multiplicity which is a factor of 2 and the sum in $\|\omega\|_1$ counts all of these.
\end{remark}



\begin{thebibliography}{10}

\bibitem{aurzada2008short}
F.~Aurzada.
\newblock A short note on small deviations of sequences of iid random variables
  with exponentially decreasing weights.
\newblock {\em Statistics \& Probability Letters}, 78(15):2300--2307, 2008.

\bibitem{AurzadaLifshits2009}
F.~Aurzada and M.~Lifshits.
\newblock On the small deviation problem for some iterated processes.
\newblock {\em Electron. J. Probab.}, 14:no. 68, 1992--2010, 2009.

\bibitem{BGT1987}
N.~H. Bingham, C.~M. Goldie, and J.~L. Teugels.
\newblock {\em Regular variation}, volume~27 of {\em Encyclopedia of
  Mathematics and its Applications}.
\newblock Cambridge University Press, Cambridge, 1987.

\bibitem{Bogachev1998}
V.~I. Bogachev.
\newblock {\em Gaussian measures}, volume~62 of {\em Mathematical Surveys and
  Monographs}.
\newblock American Mathematical Society, Providence, RI, 1998.

\bibitem{Borell1978}
C.~Borell.
\newblock Tail probabilities in {G}auss space.
\newblock In {\em Vector space measures and applications ({P}roc. {C}onf.,
  {U}niv. {D}ublin, {D}ublin, 1977), {I}}, volume 644 of {\em Lecture Notes in
  Math.}, pages 73--82. Springer, Berlin, 1978.

\bibitem{Borell1984}
C.~Borell.
\newblock On the {T}aylor series of a {W}iener polynomial.
\newblock Seminar notes on multiple stochastic integration, polynomial chaos
  and their integration. Case Western Reserve University, Cleveland, 1984.

\bibitem{borovkov2008small}
A.~Borovkov and P.~Ruzankin.
\newblock On small deviations of series of weighted random variables.
\newblock {\em Journal of Theoretical Probability}, 21(3):628--649, 2008.

\bibitem{ChenKuelbsLi2000}
X.~Chen, J.~Kuelbs, and W.~Li.
\newblock A functional {LIL} for symmetric stable processes.
\newblock {\em Annals of Probability}, 28:258--276, 2000.

\bibitem{Chung1948}
K.~L. Chung.
\newblock On the maximum partial sums of independent random variables.
\newblock {\em Transactions of the American Mathemadical Society}, 64:205--233,
  1948.

\bibitem{DriverGordina2008}
B.~K. Driver and M.~Gordina.
\newblock Heat kernel analysis on infinite-dimensional {H}eisenberg groups.
\newblock {\em Journal of Functional Analysis}, 255(2):2395--2461, 2008.

\bibitem{ikeda1989stochastic}
N.~Ikeda and S.~Watanabe.
\newblock {\em Stochastic differential equations and diffusion processes}.
\newblock Kodansha scientific books. North-Holland, 1989.

\bibitem{KuelbsLi2002}
J.~Kuelbs and W.~Li.
\newblock A functional {LIL} and some weighted occupation measure results for
  fractional {B}rownian motion.
\newblock {\em Journal of Theoretical Probability}, 15(4):1007--1030, 2002.

\bibitem{KuelbsLi2005}
J.~Kuelbs and W.~Li.
\newblock A functional {LIL} for stochastic integrals and the {L}\'{e}vy area
  process.
\newblock {\em Journal of Theoretical Probability}, 18(2):261--289, 2005.

\bibitem{Kuo1975}
H.~H. Kuo.
\newblock {\em Gaussian measures in {B}anach spaces}.
\newblock Springer-Verlag, Berlin, 1975.
\newblock Lecture Notes in Mathematics, Vol. 463.

\bibitem{Ledoux1994}
M.~Ledoux.
\newblock Isoperimetry and {G}aussian analysis.
\newblock In {\em Lectures on probability theory and statistics
  ({S}aint-{F}lour, 1994)}, volume 1648 of {\em Lecture Notes in Math.}, pages
  165--294. Springer, Berlin, 1996.

\bibitem{Li2001}
W.~V. Li.
\newblock Small ball probabilities for {G}aussian {M}arkov processes under the
  {$L_p$}-norm.
\newblock {\em Stochastic Process. Appl.}, 92(1):87--102, 2001.

\bibitem{LiShao2001}
W.~V. Li and Q.-M. Shao.
\newblock Gaussian processes: Inequalities, small ball probabilities and
  applications.
\newblock {\em Stochastic Processes: Theory and Methods. Handbook of
  Statistics}, 19:533--598, 2001.

\bibitem{lifshits1997lower}
M.~Lifshits.
\newblock On the lower tail probabilities of some random series.
\newblock {\em The Annals of Probability}, 25(1):424--442, 1997.

\bibitem{Lifshits1999}
M.~A. Lifshits.
\newblock Asymptotic behavior of small ball probabilities.
\newblock In {\em Theory of Probability and Mathematical Statistics.
  Proceedings of VII International Vilnius Conference}, pages 453--468, 1999.

\bibitem{LifshitsLinde2002}
M.~A. Lifshits and W.~Linde.
\newblock Approximation and entropy numbers of {V}olterra operators with
  application to {B}rownian motion.
\newblock {\em Mem. Amer. Math. Soc.}, 157(745):viii+87, 2002.

\bibitem{LifshitsLinde2005}
M.~A. Lifshits and W.~Linde.
\newblock Small deviations of weighted fractional processes and average
  non-linear approximation.
\newblock {\em Trans. Amer. Math. Soc.}, 357(5):2059--2079 (electronic), 2005.

\bibitem{mayer1993probability}
E.~Mayer-Wolf and O.~Zeitouni.
\newblock The probability of small gaussian ellipsoids and associated
  conditional moments.
\newblock {\em The Annals of Probability}, 21(1):14--24, 1993.

\bibitem{Nelson73c}
E.~Nelson.
\newblock The free {M}arkoff field.
\newblock {\em J. Functional Analysis}, 12:211--227, 1973.

\bibitem{Nelson73b}
E.~Nelson.
\newblock Quantum fields and {M}arkoff fields.
\newblock In {\em Partial differential equations ({P}roc. {S}ympos. {P}ure
  {M}ath., {V}ol. {XXIII}, {U}niv. {C}alifornia, {B}erkeley, {C}alif., 1971)},
  pages 413--420. Amer. Math. Soc., Providence, R.I., 1973.

\bibitem{Remillard1994}
B.~Remillard.
\newblock On {C}hung's law of the iterated logarithm for some stochastic
  integrals.
\newblock {\em The Annals of Probability}, 22(4):1794--1802, 1994.

\bibitem{rogers2000diffusions}
L.~Rogers and D.~Williams.
\newblock {\em Diffusions, {M}arkov Processes and {M}artingales: Volume 2,
  {I}t{\^o} Calculus}.
\newblock Cambridge Mathematical Library. Cambridge University Press, 2000.

\bibitem{rozovsky2011small}
L.~V. Rozovsky.
\newblock On small deviations of series of weighted positive random variables.
\newblock {\em Journal of Mathematical Sciences}, 176(2):224--231, 2011.

\bibitem{rozovskii2007small}
L.~V. Rozovsky.
\newblock Small deviation probabilities for sums of independent positive random
  variables.
\newblock {\em Zapiski Nauchnykh Seminarov POMI}, 341:151--167, 2007.

\bibitem{Shi1994}
Z.~Shi.
\newblock Liminf behaviors of the windings and {L}\'{e}vy's stochastic areas of
  planar {B}rownian motion.
\newblock {\em Seminaire de probabilities (Strasbourg)}, 28:122--137, 1994.

\bibitem{ZhangLin2006}
R.~Zhang and Z.~Lin.
\newblock A functional {LIL} for {$m$}-fold integrated {B}rownian motion.
\newblock {\em Chinese Ann. Math. Ser. B}, 27(4):459--472, 2006.

\bibitem{ZhangLin2010}
R.~M. Zhang and Z.~Y. Lin.
\newblock A functional {LIL} for integrated {$\alpha$} stable process.
\newblock {\em Acta Math. Sin. (Engl. Ser.)}, 26(2):393--404, 2010.

\end{thebibliography}




\end{document}